\documentclass[11pt]{article}
\usepackage{tikz}

\usepackage[a4paper, margin=2.3cm]{geometry}
\usepackage{amsmath,amssymb,amsthm}
\usepackage{verbatim}
\usepackage{color}
\usepackage{parskip}
\usepackage{hyperref}
\usepackage[capitalize]{cleveref}
\usepackage{setspace}
\usepackage{graphicx}
\usepackage{mathtools}
\usepackage[thinc]{esdiff}
\usepackage{pgfplots}
\pgfplotsset{compat=1.18}
\usepgfplotslibrary{fillbetween}
 
\newtheorem{theorem}{Theorem}[section]

\newtheorem{corollary}[theorem]{Corollary}
\newtheorem{lemma}[theorem]{Lemma}
\newtheorem{proposition}[theorem]{Proposition}

\newtheorem{conjecture}[theorem]{Conjecture}
\newtheorem*{definition*}{Definition}

\def\thang#1{\noindent
\textcolor{blue}
{\textsc{(Thang:}
\textsf{#1})}}
\def\ben#1{\noindent
\textcolor{red}
{\textsc{(Ben:}
\textsf{#1})}}

\def\F{\mathcal{F}}

\def\E{\mathrm{E}}

\def \F{{\mathbb F}}

\begin{document}
\title{Orthogonal projections and incidence bounds over prime fields}
\author{Ben Lund \thanks{Institute for Mathematics and Interdisciplinary Sciences, Xidian University. Email: {\tt lund.ben@gmail.com}}\and Thang Pham\thanks{Institute for Mathematics and Interdisciplinary Sciences, Xidian University. \newline
\hspace*{0.6cm} Email: {\tt thangpham.math@gmail.com}}\and Le Anh Vinh \thanks{Vietnam Institute of Educational Sciences. Email: {\tt vinhle@vnies.edu.vn}}}
\date{}
\maketitle
\begin{abstract}
Let $p$ be an odd prime and let $E\subset \mathbb{F}_p^2$ with $|E|=p^a$, where $0<a\le 1$.
For a direction $V$ (a $1$-dimensional subspace of $\mathbb{F}_p^2$), let $\pi^V:\mathbb{F}_p^2\to \mathbb{F}_p^2/V$ denote the quotient map.
We bound the size of the exceptional set of directions for which the projection $\pi^V(E)$ is small.
More precisely, for $a/2\le s\le a$ define
\[
T_s^{1,2}(E):=\{V\in G(1,\mathbb{F}_p^2):\ |\pi^V(E)|<p^s\}.
\]
We prove
\[
|T_s^{1,2}(E)|\ll \min\{\,p^{\frac52 s-a},\ p^{6s-3a},\ p^s\,\},
\]
which improves the best previously known estimates over prime fields in the range $a/2\le s<2a/3$, and yields the first substantial progress toward Chen’s 2018 conjecture. The key new ingredient is a novel point-line incidence bound, of independent interest, that yields a power saving when the line set spans only moderately many distinct directions. In the reverse direction, we also obtain an incidence estimate for Cartesian products $A\times B$ with line families
$\{y=ax+b:\ a,b\in C\}$ with explicit dependence on the additive energy $E^+(C)$. We also discuss connections to the sum-set problem and the distinct dot-product values conjecture.

\end{abstract}
\tableofcontents

\section{Introduction}
In this paper, we study orthogonal projections in vector spaces over finite fields, with an emphasis
on the planar case over prime fields. Projection theorems quantify how a set looks from different
directions. In the Euclidean plane, the classical Marstrand-Mattila projection theorem shows that,
outside a small exceptional set of directions, orthogonal projections behave as large as they should,
see \cite{Mar1,Mat2} and also \cite{bay,tam,haitu,1gan,K.O.V.25} for further results and recent progress.
In the finite field plane $\mathbb{F}_q^2$, cardinality plays the role of dimension, and one asks:
given $E\subset \mathbb{F}_q^2$ with $|E|=q^a$, for how many directions $V$ can the projection $\pi^V(E)$
be unusually small? Here and throughout the paper, $q$ is an odd prime power.



Let $G(1,\mathbb{F}_q^2)$ denote the Grassmannian of $1$-dimensional linear subspaces of $\mathbb{F}_q^2$. Given $E\subset \mathbb{F}_q^2$, the orthogonal projection of $E$ along $V \in G(1, \F_q^2)$ is defined to be 
\[\pi^V(E):=\{x+V\colon (x+V)\cap E\ne\emptyset,~x\in \mathbb{F}_q^2\}.\]
Equivalently, $|\pi^V(E)|$ is the number of affine lines parallel to $V$ that intersect $E$.

Let $E\subset \mathbb{F}_q^2$ be a set with $|E|=q^a$ $(0<a<2)$. For $s\in (0, a)$, define
\begin{align*}
    T_s^{1, 2}(E)&=\{ V\in G(1, \mathbb{F}_q^2):|\pi^V(E)|<q^s \}
\end{align*}
to be the set of directions $V$ for which the projection $\pi^V(E)$ has size less than $q^s$.

Using discrete Fourier analysis, Chen \cite{chen} established the following projection theorem in all dimensions over prime fields, and his method extends with minor modifications to arbitrary finite fields. Moreover, it was recently shown by Ham, Hoang, Hung, Koh, and Pham in \cite{koh} that Chen’s bound is sharp in general for dimensions $d\ge 3$. Since this paper is concerned with the planar case, we state only the two-dimensional version below. \footnote{Throughout the paper, the notation $\ll$ is used to absorb a multiplicative constant, not depending on any other parameters, the notation $\lesssim$ is used to absorb an arbitrarily small polynomial factor, {\em i.e.} $f(q) \lesssim g(q)$ means that, for any $\epsilon > 0$, there is a constant $C_\epsilon > 0$ such that $f(q) \leq C_\epsilon q^{\epsilon}g(q)$ for all $q>0$.}

\begin{theorem}[Chen, \cite{chen}]\label{thm1-chen}
Let $E\subset \mathbb{F}_q^2$ with $|E|=q^a$.
\begin{enumerate}
    \item If $0<s<1<a$, then 
    \begin{equation}\label{eq:Chen1}|T_s^{1, 2}(E)|\ll q^{1+s-a}.\end{equation}
    \item If $0<s<a<1$, then 
    \begin{equation}\label{eq:Chen2}|T_s^{1, 2}(E)|\ll q^s.\end{equation}
\end{enumerate}
\end{theorem}

The upper bound of $q^s$ is tight when $q^s$ is the size of a sub-field of $\F_q$, and $a=2s$. Outside such subfield-type configurations, the sharpness of the $q^s$ bound is less clear, and it would be interesting to determine when it can be improved and to describe near-extremizers.

Bright and Gan \cite{BG} obtained the following refinement of Chen's bound. 

\begin{theorem}[Bright--Gan, \cite{BG}]\label{BG-r} If $s<a/2$, then 
    \[|T_{s}^{1, 2}(E)|\lesssim q^{1+2(s-a)}.\]
    
In particular, if $|E|=q^{a}$, $a\ge 1$, and  $0<s<\frac{a}{2}$, then
    $|T_s^{1, 2}(E)|\lesssim 1.$
\end{theorem}
As noted in \cite{BG}, this result is sharp in its range, i.e. $s<\frac{a}{2}$. 

Our main result in this paper concerns orthogonal projections in the two-dimensional plane over a prime field. Before stating the result, we discuss the following conjecture due to Chen \cite{chen}.

\begin{conjecture}\label{chen-con}
Let $a/2\le s\le a \le 1$, let $p$ be an odd prime, and let $E\subset \mathbb{F}_p^2$ with $|E|=p^a$. 
Then 
\[|T_s^{1, 2}(E)|\ll p^{2s-a}.\]
\end{conjecture}

There are two pieces of evidence supporting Conjecture~\ref{chen-con}.
First, Bright and Gan \cite{BG} constructed sets $E\subset \mathbb{F}_p^2$ for which the exponent $2s-a$ is attained.
Second, Orponen \cite{Or} observed that in the Euclidean plane $\mathbb{R}^2$, the Szemer\'edi-Trotter theorem \cite{Szemeredi} implies
\[
\bigl|\{V:\ |\pi^V(E)|\le N\}\bigr|\ll \frac{N^2}{|E|}\qquad\text{for all }N\ll |E|,
\]
for every finite set $E\subset\mathbb{R}^2$.

An analogous statement over $\mathbb{F}_p^2$, which confirms Conjecture \ref{chen-con}, would follow from the widely conjectured Szemer\'edi-Trotter type incidence bound over prime fields: for any sets of points $P$ and lines $L$ in $\mathbb{F}_p^2$,
\begin{equation}\label{eq:FFSTconj}
I(P,L)\ \ll\ \frac{|P||L|}{p}\ +\ |P|^{2/3}|L|^{2/3}\ +\ |P|\ +\ |L|.
\end{equation}


The hypothesis that $p$ is prime is necessary in \cref{chen-con}, since, as discussed earlier, the bound (\ref{eq:Chen2}) is best possible over arbitrary finite fields due to sub-fields, at least for certain values of $s,a$.

Establishing this conjecture remains a challenge in the area.
One important implication of this conjecture is that if $a/2\le s<a \le 1$, then there exists $\epsilon=\epsilon(a, s)>0$ such that 
\begin{equation}\label{weakform} |T_s^{1, 2}(E)|=|\{V\in G(1, \mathbb{F}_p^2)\colon |\pi^V(E)|\le p^{s}\}|\le p^{s-\epsilon}.\end{equation}
The continuous version of \cref{weakform} has been proved in \cite{OS} by Orponen and Shmerkin recently.

Our main result is the following, which makes the first substantial progress toward \cref{chen-con}.

\begin{theorem}\label{2dimension}
Let $p$ be an odd prime, $a/2 \leq s\le a \le 1$, and let $E\subset \mathbb{F}_p^2$ with $|E|=p^a$. We have 
\[|T_s^{1, 2}(E)|\ll \min\left\lbrace p^{6s-3a},~p^{\frac{5s}{2}-a}, ~p^s \right\rbrace.\]

\end{theorem}

\begin{figure}[h!]
  \centering
  \includegraphics[width=0.9\textwidth]{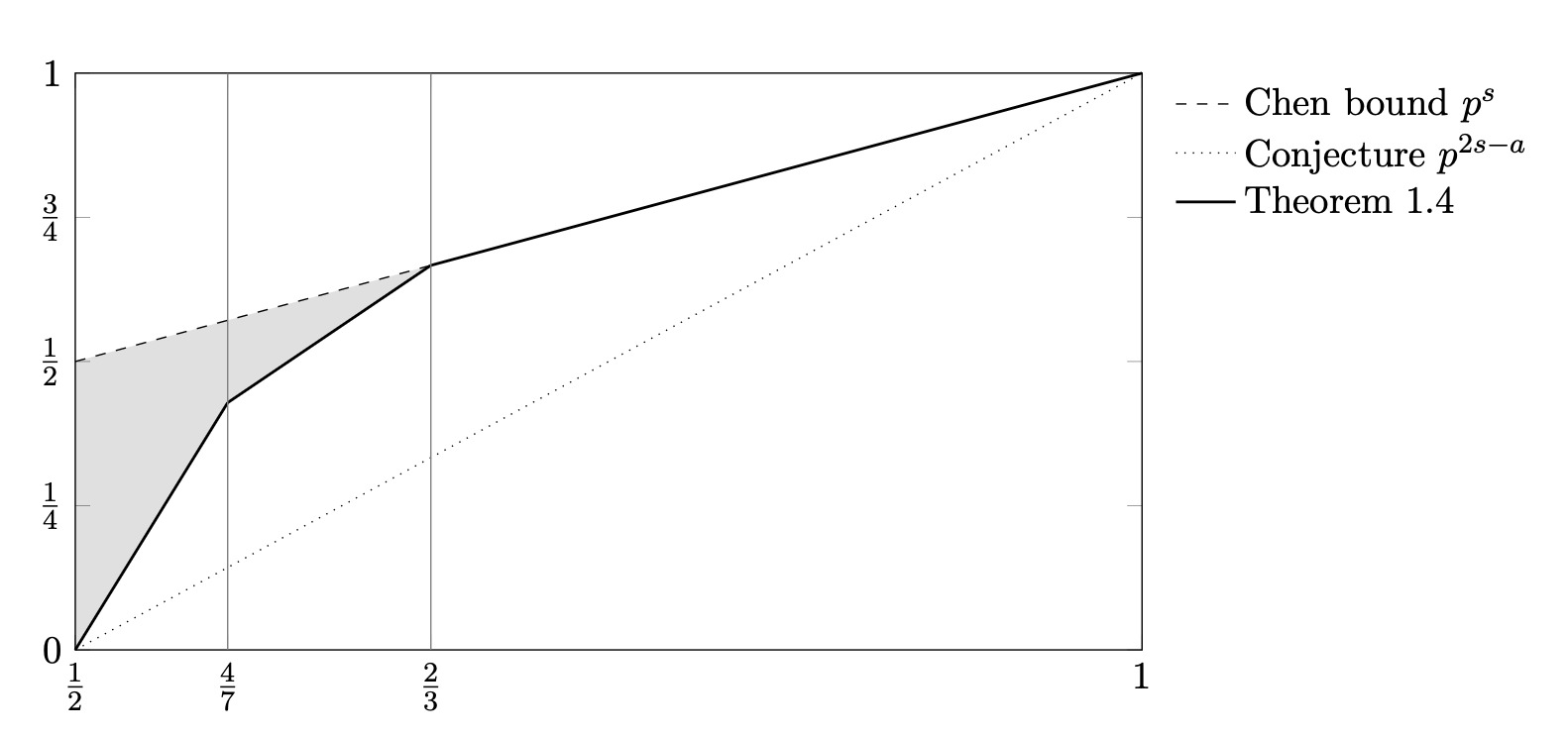}
  \caption{$x=\frac{s}{a}$, $y=\frac{1}{a}\log_p |T^{1,2}_s(E)|$}  
  \label{fig:fixed1}
\end{figure}

In the range $a/2\le s<2a/3$, \Cref{2dimension} yields a strict improvement over the best previously known bound provided by \cref{thm1-chen}. More precisely, \Cref{2dimension} gives
\[
|T^{1,2}_s(E)|
\ \ll\
\begin{cases}
p^{6s-3a}, & a/2\le s\le 4a/7,\\[2pt]
p^{\frac52 s-a}, & 4a/7\le s\le 2a/3,\\[2pt]
p^{s}, & 2a/3\le s\le a.
\end{cases}
\]
Therefore, for $s<2a/3$, \Cref{2dimension} implies \cref{weakform} with
\[
\epsilon \;\ge\; \max\{\,3a-5s,\ a-\tfrac32 s\,\}\;>\;0.
\]


Although Chen’s result is sharp in general for dimensions $d \ge 3$, several partial improvements have recently been obtained for specific ranges of the parameters $s$ and $a$. We refer the reader to Fraser and Rakhmonov~\cite{JFraser} and to Rose~\cite{AR} for further results and discussion.

In order to prove \cref{2dimension}, we develop a novel incidence bound, which is of independent interest, between an arbitrary set of points and a collection of lines that point in relatively few distinct directions.
Here, we briefly describe how to reduce the projection problem to a question about point-line incidences.
Suppose that $E$ is a set of points in the plane.
Let $L$ be the set of affine lines such that each $\ell \in L$ is incident to at least one point of $E$, and for each $\ell \in L$, there is a $V \in T_s^{1,2}(E)$ such that $\ell$ is parallel to $V$.
Then, the number of incidences between $E$ and $L$ is equal to $|T_s^{1,2}(E)|\,|E|$.
Hence, a suitable upper bound on $I(E,L)$ translates to an upper bound on $|T_s^{1,2}(E)|$.
Note that we could obtain a bound by using the best known general incidence bound in the plane due to
Stevens and de Zeeuw \cite{frank}. Concretely, writing $m:=|E|=p^{a}$ and
$n:=|L|\le |T_s^{1,2}(E)|p^{s}$, their theorem in the range $m^{7/8}<n<m^{8/7}$ yields
\[
I(E,L)\ll m^{11/15}n^{11/15},
\]
and since $I(E,L)=|T_s^{1,2}(E)||E|$ this gives
\[
|T_s^{1,2}(E)|\ll p^{\frac{11}{4}s-a}.
\]
However, this bound does not exploit that the lines in $L$ point in only $|T_s^{1,2}(E)|$ directions, and the following theorem improves it in the direction-sparse regime relevant to \cref{2dimension}.

\begin{theorem}\label{th:dualHalfProductIncidenceBound}
    Let $p$ be an odd prime, $E$ be a set of points in $\F_p^2$, and $L$ be a set of lines in $\F_p^2$.
    Suppose that each line of $L$ lies in one of $t \ll p^2|E|^{-1}$ distinct directions.
    Then,
    \[I(E,L) \ll \min(t^{3/10}|L|^{3/5}|E|^{7/10}, t^{2/9}|L|^{5/9}|E|^{7/9}) + t^{1/2}\,|L|^{1/2}\,|E|^{1/2} + |L|^{2/3}|E|^{2/3} + |L| + |E|.\]
\end{theorem}

The hypothesis $t \ll p^{2}|E|^{-1}$ in Theorem \ref{th:dualHalfProductIncidenceBound} is imposed in order to apply the Stevens-de Zeeuw Cartesian product incidence
bound (Theorem \ref{th:productIncidenceBound} below), which requires $|A||L|\ll p^{2}$. In the projection applications in this paper, we always have $|E|\le p$, hence $p^{2}|E|^{-1}\ge p$. Since $t\le p+1$ in $\mathbb{F}_p^2$, the condition $t\ll p^{2}|E|^{-1}$ is satisfied up to the absolute implied constant.

By projective point-line duality, \Cref{th:dualHalfProductIncidenceBound} is equivalent to the following
formulation (Theorem \ref{th:halfProductIncidenceBound}), which some readers may find more geometric.

\begin{theorem}\label{th:halfProductIncidenceBound11}
Let $p$ be an odd prime, let $Z$ be a set of points in $\F_p^2$, and let $L$ be a set of lines in $\F_p^2$.
Suppose that $Z$ is contained in the union of $y\ll p^2|L|^{-1}$ distinct horizontal lines.
Then,
\[
I(Z,L)\ll
\min\bigl(y^{3/10}|Z|^{3/5}|L|^{7/10},\ y^{2/9}|Z|^{5/9}|L|^{7/9}\bigr)
+y^{1/2}|Z|^{1/2}|L|^{1/2}
+|Z|^{2/3}|L|^{2/3}
+|Z|+|L|.
\]
\end{theorem}

\paragraph{Main ideas of the proof of Theorem~\ref{th:halfProductIncidenceBound11}.}
After discarding horizontal lines and lines incident to at most one point of $Z$, we regularize the
incidence bipartite graph to obtain subsets $Z_1\subset Z$ and $L_1\subset L$ with comparable incidence
density and uniform degree lower bounds. Choosing a point $p\in Z_1$ incident to many lines of $L_1$,
we consider the set $Z_p$ of points of $Z_1$ lying on the pencil of lines through $p$ (excluding $p$).
This gives a strong lower bound for $I(Z_p,L_1)$.
On the other hand, $Z_p$ is contained in the union of that pencil together with at most $y$ horizontal lines.
A projective transformation sending the horizontal line through $p$ to the line at infinity (while preserving
horizontality) maps the pencil to a family of parallel lines, so the image of $Z_p$ lies in a grid $X\times Y$
with $|Y|\le y$. Applying the Stevens-de Zeeuw Cartesian-product incidence bound to $(X\times Y,L_1)$
(in two cases depending on whether $|X|\le |Y|$ or $|Y|\le |X|$) yields either the mixed term, the
Szemer\'edi-Trotter-type term, or one of the two direction-sensitive exponents. Reintroducing the discarded
lines completes the proof.

Regarding Conjecture \ref{chen-con}, it is worth noting that the following direction-sensitive point-line incidence conjecture, which is weaker than (\ref{eq:FFSTconj}), would be sufficient

\begin{equation}
I(P,L)\ \ll\ \frac{|P||L|}{p}\ +\ |P|^{2/3}|L|^{2/3}\ +t^{1/2}|P|^{1/2}|L|^{1/2}+\ |P|\ +\ |L|.
\end{equation}
Substituting $t=p+1$ essentially gives us the incidence bound obtained by the third listed author in \cite{vinh1}. 

While Theorem \ref{th:dualHalfProductIncidenceBound} and its duality Theorem \ref{th:halfProductIncidenceBound11} were developed to prove the projection theorem (Theorem \ref{2dimension}), in this paper, we show that projection-type arguments can also be used to prove incidence
bounds for structured configurations. Theorem \ref{AAincidence} below presents such an application, which is of independent interest.




Given $C\subset \mathbb{F}_p$, we say a set $L$ of lines is of the form $C\times C$ if all lines in $L$ are of the form $y=ax+b$ with $a, b\in C$. The additive energy of $C$, denoted by $E^+(C)$, is defined by 
\[E^+(C):=|\{(c_1, c_2, c_3, c_4)\in C^4\colon c_1+c_2=c_3+c_4\}|.\]

\begin{theorem}\label{AAincidence}
Let $A, B, C\subset \mathbb{F}_p$ with $|A|\le |B|$, $|C||B|^2\le E^+(C)$, and $|A||C|^2=O(p^2)$. Assume that $L$ is of the form $C\times C$, then 
\[I(A\times B, L)\ll |A|^{7/8}|B|^{1/2}|L|^{3/8}E^+(C)^{1/4}.\]
\end{theorem}

Theorem \ref{AAincidence} is proved in the reverse direction: a projection-type expansion statement
(Theorem \ref{2.4}) is a key input in its proof.

To see how good this theorem is, we first compare it with a result of Stevens and De Zeeuw \cite{frank} (Theorem \ref{th:productIncidenceBound}). A direct computation shows that Theorem \ref{AAincidence} is better in the range 
\begin{equation}\label{eq:11}
    E^+(C)^2\ll \frac{|C|^6}{|A|}.
\end{equation}
Under the assumption that $L$ is of the form $C\times C$, we also know from a result of Rudnev and Shkredov \cite{RSH}  that
\[I(A\times B, L)\ll |A|^{5/8}|B|^{1/2}|C|^{\frac{27}{16}}+|B|^{1/2}|C|^2\sqrt{\max\{1, |A|^2p^{-1}\}}.\]
Compared to Theorem \ref{AAincidence}, assume $|A|\ll p^{1/2}$ for simplicity, Theorem \ref{AAincidence} is better in the range 
\begin{equation}\label{eq:22}E^+(C)^4\ll \frac{|C|^{15}}{|A|^4}.\end{equation}
Putting (\ref{eq:11}) and (\ref{eq:22}) together, Theorem \ref{AAincidence} is effective in the range 
\[|C||B|^2\le E^+(C)\ll \min \left\lbrace \frac{|C|^3}{|A|^{1/2}}, ~\frac{|C|^{15/4}}{|A|}\right\rbrace.\]

We end the introduction with two remarks. 

First, the finite field incidence problems have found a series of applications in number theory, restriction theory, discrete geometry, additive combinatorics, and theoretical computer
science. We refer the reader to \cite{AKL,
KLP22, Le15, Lewko19, MT04, MSS1, Rudnev, RSH} and references therein for more details. In a broader context, looking beyond the specific ambient setting to focus on the geometric nature of the operations, we observe that sumsets are essentially orthogonal projections of Cartesian products. For a subset $A$ of an abelian group, the sumset $A+A$ corresponds to the projection of the product set $E = A \times A$ along the direction $V=\langle(1,-1)\rangle$. In this geometric language, the restricted sum-set problem, concerning $A\dot{+}A$, asks for the size of the projection of the off-diagonal subset $E \setminus \Delta$ along this fixed direction. The density of this restricted image was the subject of the Erd\H{o}s-Graham conjectures, resolved by Hegyv\'{a}ri, Hennecart, and Plagne \cite{HHP}. While their work establishes that the projection of a specific structured set remains large along a fixed direction despite the removal of the diagonal, the present paper addresses the global distribution of projection sizes for arbitrary sets $E$ as the direction $V$ varies.

Second, projection exceptional set estimates have consequences for dot-product problems.  The
distinct dot-product values conjecture \cite[Conjecture~4]{MPRRS} says that for every
\(E\subset \mathbb{F}_p^2\) with \(|E|\le p\),
\(|E\cdot E|\gg |E|^{1-\epsilon}\) for all \(\epsilon>0\). The current record is
\(|E\cdot E|\gg |E|^{108/161}\) from \cite{MPRRS}.  Using Theorem~\ref{2dimension}, one obtains a
strong pinned result in an unbalanced setting.  Namely, let \(E,F\subset \mathbb{F}_p^2\) with
\(|E|=p^a\le p\), and assume that each line through the origin contains at most \(O(1)\) points of
\(F\) (so \(F\) determines \(\gg |F|\) distinct directions).  In the range, $
|E|^{3/7}\ll |F|\ll |E|^{2/3},$
Theorem~\ref{2dimension} implies that there exists \(u\in F\) such that
\[
|u\cdot E|\gg (|E||F|)^{2/5}.
\]
For a general set \(F\), a standard dyadic pigeonholing
argument produces a subset \(F'\subset F\) for which each line through the origin contains
\(\sim m\) points of \(F'\), so that\footnote{By $\mathrm{Dir}(F')$ we mean the number of lines passing through the origin and intersecting $F'$} \(|\mathrm{Dir}(F')|\sim |F'|/m\).  One may then combine the
preceding projection argument applied to \(\mathrm{Dir}(F')\) with the trivial contribution coming
from a large radial fiber to obtain lower bounds for \(|E\cdot F|\) after optimizing in \(m\).
We do not pursue this direction here.

\section{Proof of incidence theorems (Theorems \ref{th:dualHalfProductIncidenceBound} and \ref{AAincidence})}
\subsection{Proof of Theorem \ref{th:dualHalfProductIncidenceBound}}
In this section, we prove the following dual form of \cref{th:dualHalfProductIncidenceBound}.

\begin{theorem}\label{th:halfProductIncidenceBound}
    Let $p$ be prime, let $Z$ be a set of points in $\F_p^2$, and let $L$ be a set of lines in $\F_p^2$.
    Suppose that the points of $Z$ are contained in the union of $y \ll p^2|L|^{-1}$ distinct horizontal lines.
    Then,
    \[I(Z,L) \ll  \min(y^{3/10}|Z|^{3/5}|L|^{7/10}, y^{2/9}|Z|^{5/9}|L|^{7/9}) + y^{1/2}\,|Z|^{1/2}\,|L|^{1/2} + |Z|^{2/3}|L|^{2/3} + |Z| + |L|.\]
\end{theorem}


The key tool in the proof of \cref{th:halfProductIncidenceBound} is the following theorem of Stevens and de Zeeuw \cite{frank}.
\begin{theorem}[Stevens, de Zeeuw]\label{th:productIncidenceBound}
Let $p$ be a prime, let $A
,B \subset \F_p$ and $L$ a set of lines in $\F_p^2$.
Suppose that $|A| \leq |B|$ and $|A|\,|L| \ll p^2$.
Then,
\[ I(A \times B, L) \ll |A|^{3/4} |B|^{1/2} |L|^{3/4} + |A|\,|B| + |L|.\]
\end{theorem}

In \cite{frank}, the authors use \cref{th:productIncidenceBound} to prove a quantitatively weaker bound for arbitrary sets of points and lines, following an argument introduced by Bourgain, Katz, and Tao \cite{BKT}.
Here, we adapt their method to prove \cref{th:halfProductIncidenceBound}.

We first establish the following simple graph theoretic lemma.
We will use it to remove points and lines that are involved in an unusually small number of incidences.

\begin{lemma}\label{th:pruning}
For any bipartite graph $G=(L,R,E)$, there is an induced subgraph $G'=(L',R',E')$ of $G$ such that 
\begin{align*}\min_{v \in L'}\deg(v)\min_{u \in R'} \deg(u) &> \frac{|E|^2}{4|L|\,|R|} \text{, and }\\
\frac{|E'|^2}{|L'|\,|R'|} &\geq \frac{|E|^2}{|L|\,|R|}.
\end{align*}
\end{lemma}
\begin{proof}
    We proceed by induction on $|L \sqcup R|$.
    Suppose that the conclusion doesn't hold, and without loss of generality suppose that $v \in L$ such that $\deg(v) \leq 2^{-1}|E|\,|L|^{-1}$. 
    Let $G_1=(L_1,R_1,E_1)$ be the subgraph of $G$ induced on $(|L| \setminus \{v\}) \sqcup R$.
    We have
    \begin{align*}
        \frac{|E_1|^2}{|L_1|\,|R_1|} & = \frac{(|E|- \deg(v))^2}{(|L|-1)\,|R|}\\
        &\geq \frac{(|E|- 2^{-1}|E|\,|L|^{-1})^2}{(|L|-1)\,|R|} \\
        &\geq \frac{|E|^2}{|L|\,|R|} \frac{(1-2^{-1}|L|^{-1})^2}{1 - |L|^{-1}} \\
        &> \frac{|E|^2}{|L|\,|R|}.
    \end{align*}
    The conclusion of the lemma follows by induction.
\end{proof}

Now we are ready to proceed with the proof of \cref{th:halfProductIncidenceBound}.

\begin{proof}[Proof of Theorem \ref{th:halfProductIncidenceBound}]
    First, we remove the horizontal lines of $L$, and those that contain only one point of $Z$.
    Let $L'$ be the lines of $L$ that are not horizontal and each contain at least two points of $Z$.
    Since each point of $Z$ is incident to at most one horizontal line of $L$, we have $I(Z,L') \geq I(Z,L)-|Z|-|L|$.
    If $I(Z,L) \leq 2(|Z| + |L|)$, then we're done. 
    Hence, we may assume that $I(Z,L') \geq 2^{-1}I(Z,L)$.
    
    By \cref{th:pruning}, there are $Z_1 \subset Z$ and $L_1 \subset L'$ such that
    \begin{align}
        (\min_{p \in Z_1}\#\{\ell \in L_1: p \in \ell\})(\min_{\ell \in L_1}\#\{p \in Z_1: p \in \ell\}) &> \frac{I(Z,L')^2}{4|Z|\,|L'|}, \text{ and} \label{eq:aveLpZl} \\
        \frac{I(Z_1,L_1)^2}{|L_1|\,|Z_1|} &\geq \frac{I(Z,L')^2}{|Z|\,|L'|}. \label{eq:I1LowerBound}
    \end{align}
    For each $z \in Z_1$, denote $L_z = \{\ell \in L_1: z \in \ell\}$.
    Likewise, for each $\ell \in L_1$, denote $Z_\ell = \{z \in Z_1: z \in \ell\}$.
    Denote $I = I(Z,L')$, and $I_1 = I(Z_1, L_1)$.
    
    Let $p \in Z_1$ be a point such that $|L_p| \geq I_1|Z_1|^{-1}$.
    There must be such a point since $I_1|Z_1|^{-1}$ is the average number of lines of $L_1$ incident to each point of $Z_1$.
    Using \cref{eq:I1LowerBound}, we can bound $|L_p|$ as
    \begin{equation}\label{eq:LpLowerBound}|L_p| \geq \frac{I_1}{|Z_1|} \geq \frac{|L_1|}{I_1}\frac{I_1^2}{|Z_1|\,|L_1|} \geq \frac{|L_1|}{I_1} \frac{I^2}{|Z|\,|L'|} \geq \frac{|L_1|}{|L'|} \frac{I}{|Z|}.\end{equation}

    Let $Z_p = \bigcup_{\ell \in L_p} (\ell \cap Z_1 \setminus \{p\})$.
    Since each pair of lines in $L_p$ intersects at $p$,
    \begin{equation}\label{eq:ZpBound}
        |Z_p| \geq \sum_{\ell \in L_p} (|Z_\ell|-1) \geq \frac{1}{2} \sum_{\ell \in L_p} |Z_\ell|.
    \end{equation}
    
    We will bound $I_p = I(Z_p,L_1)$ from below and above.
    For the lower bound, we use \cref{eq:aveLpZl,eq:ZpBound} to obtain
    \begin{equation}\label{eq:ZpLLowerBound}
        I_p = \sum_{z \in Z_p}|L_z| \geq |Z_p|\min_{z \in Z_1}|L_z| \geq \frac{1}{2}|L_p|\min_{\ell \in L_1}|Z_\ell|\min_{z \in Z_1}|L_z| \geq \frac{|L_p|I^2}{8|Z||L'|}.
    \end{equation}
    
    We will use \cref{th:productIncidenceBound} for the upper bound on $I_p$.
    In order to apply \cref{th:productIncidenceBound}, we use a projective transformation to map $Z_p$ to a grid.
    In more detail, let $Y$ be the set of at most $y$ horizontal lines that contain the points of $Z_p$.
    Then the points of $Z_p$ are contained in the intersections of lines of $Y$ with the lines $L_p$ through $p$.
    We can map the horizontal line through $p$ (which is not a line of $L_p$) to the line at infinity in such a way that the horizontal lines remain horizontal, and the lines of $L_p$ become vertical.
    Then, the image of $Z_p$ under this map is contained in a Cartesian product of size $y\times |L_p|$.

    One hypothesis of \cref{th:productIncidenceBound} is that $|A| \leq |B|$, and so we consider two cases: either $y \leq |L_p|$ or $|L_p| \leq y$.
    In the case that $y \leq |L_p|$,  \cref{eq:ZpLLowerBound,th:productIncidenceBound} imply that
    \begin{equation} \label{eq:smallYinitial}\frac{|L_p|I^2}{|Z|\,|L'|} \ll I_p \ll y^{3/4}|L_p|^{1/2}|L_1|^{3/4} + y\,|L_p| + |L_1|. \end{equation}

    If the right hand side is dominated by $y\,|L_p|$, then \cref{eq:smallYinitial} quickly simplifies to  $I \ll y^{1/2}\,|Z|^{1/2}|L|^{1/2}$.
    If the right hand side is dominated by $|L_1|$, then \cref{eq:smallYinitial,eq:LpLowerBound} together imply that $I \ll |Z|^{2/3}|L|^{2/3}$.
    If the right hand side is dominated by $y^{3/4}|L_p|^{1/2}|L_1|^{3/4}$, then we have
    \begin{align*}
    y^{3/4}|L_1|^{3/4} &\gg \frac{|L_p|^{1/2}I^2}{|Z||L'|} \geq \frac{I^{5/2}|L_1|^{1/2}}{|L'|^{3/2}|Z|^{3/2}} , \text{ hence}\\
    I^{5/2} &\ll |L'|^{3/2}|L_1|^{1/4}y^{3/4}|Z|^{3/2}.
    \end{align*}
    This leads immediately to the claimed bound $I \ll |L|^{7/10}y^{3/10}|Z|^{3/5}$.

    In the case that $|L_p| \leq y$, \cref{eq:ZpLLowerBound,th:productIncidenceBound} imply that
    \begin{equation} \label{eq:smallLpInitial}\frac{|L_p|I^2}{|Z|\,|L'|} \ll I_p \ll |L_p|^{3/4}y^{1/2}|L_1|^{3/4} + y\,|L_p| + |L_1|. \end{equation}
    The terms $y\,|L_p|$ and $|L_1|$ are handled exactly as before.
    If the right hand side of \cref{eq:smallLpInitial} is dominated by $|L_p|^{3/4}y^{1/2}|L_1|^{3/4}$, then we have
    \begin{align*}
    y^{1/2}|L_1|^{3/4} &\gg \frac{|L_p|^{1/4}I^2}{|Z||L'|} \geq \frac{I^{9/4}|L_1|^{1/4}}{|L'|^{5/4}|Z|^{5/4}} , \text{ hence}\\
    I^{9/4} &\ll |L'|^{5/4}|L_1|^{1/2}y^{1/2}|Z|^{5/4} \ll |L|^{7/4}y^{2/4}|Z|^{5/4}.
    \end{align*}
    This leads immediately to the claimed bound $I \ll y^{2/9}|Z|^{5/9}|L|^{7/9}$.

    That we can take the minimum of the two bounds follows from the observation that the upper bound of \cref{th:productIncidenceBound} has a worse dependence on the smaller of $|A|$ and $|B|$.
\end{proof}

Here is an alternate form of \cref{th:halfProductIncidenceBound}, which may be more useful for certain applications.
We omit the proof, since it is completely routine.
\begin{corollary}
Let $0 < s < 1$ and $\tau > 0$.
If $Y \subset \F_p$ with $|Y| = p^\tau$, and for each $y \in Y$, $X_y \subset \F_p$ with $|X_y| = p^s$, and each point of $Z = \bigcup_{y \in Y}X_y \times \{y\}$ is $p^s$-rich with respect to a set $L$ of lines, then
\[ |L| \gg p^{2s + \tau/7}.\]
\end{corollary}
\subsection{Proof of Theorem \ref{AAincidence}}
Given $Y\subset \mathbb{F}$ and $E\subset \mathbb{F}^2$, in this section, we study the lower bounds of the set 
\[Y\cdot E:=\{ay+b\colon (a, b)\in E, y\in Y\}.\]
More precisely, we are interested in the question of finding conditions on $E$ and $Y$ such that there exists $y\in Y$ such that $|y\cdot E|\gg |E|^{\frac{1}{2}+\epsilon}$ for some $\epsilon>0$. 

In the case that $E$ is a Cartesian product, this problem has been studied intensively in the literature with applications to sum-product problems. We refer the interested reader to \cite{Bou, MP, Rudnev} and references therein for more details.

\subsubsection*{Results over arbitrary finite fields}
\begin{theorem}\label{thm:generalProjectionBound}
    Let $Y \subseteq \mathbb{F}_q$ and $E\subseteq \mathbb{F}_q^2$. Then there exists $y\in Y$ such that 
    \[|\{ay+b\colon (a, b)\in E\}|\gg \min \left\lbrace  |Y|^{1/2}q^{1/2}, ~\frac{|Y|^{1/2}|E|}{q}, ~|E|\right\rbrace.\]
\end{theorem} 
\begin{proof}
    Define
    \[R=\{((a,b),(a',b'),y) \in E \times E \times Y : ay+b = a'y+b'\}. \]
    We will bound $|R|$ from below and from above.

    Here comes the lower bound.
    Let $M$ be the least integer such that $|\{ay + b: (a,b) \in E\}| \leq M$ for each $y \in Y$.
    By convexity,
    \begin{equation}\label{eq:RLowerBound}|R| = \sum_{y \in Y} \sum_{x \in \mathbb{F}_q} |\{(a,b) \in E: ay + b = x\}|^2 \geq |Y|\,|E|^2M^{-1}. \end{equation}

    Now the upper bound.    
    If $((a,b),(a',b'),y) \in R$ and $a = a'$, then $b=b'$.
    Hence, there are $|Y|\,|E|$ elements of $R$ such that $a = a'$.
    Let $R'$ be the set of elements of $R$ for which $a \neq a'$, and note that
    \[R' = \{((a,b),(a',b'),y) \in E \times E \times Y: a \neq a', y = (b-b')(a-a')^{-1}\}. \]
    Using the Cauchy-Schwarz inequality,
    \begin{align}
        \nonumber |R'| &= \sum_{y \in Y}|\{((a,b),(a',b')) \in E\times E: a\ne a', ~ y = (b-b')(a-a')^{-1}\}| \\ \label{eq:rprimebound} &\leq |Y|^{1/2} \left(\sum_{\lambda \in \mathbb{F}_q} |\{((a,b),(a',b')) \in E\times E: a\ne a', ~\lambda = (b-b')(a-a')^{-1}\}|^2\right)^{1/2}.
    \end{align}
    For each $(a,b) \in E$ and $\lambda \in \mathbb{F}_q$, denote the number of $(c,d) \in E \setminus \{(a,b)\}$ such that $(d-b)/(c-a) = \lambda$ by $N_{a,b}(\lambda)$.
    For a fixed $\lambda \in \F_q$,
    \[|\{((a,b),(a',b')) \in E\times E: a\ne a', ~~ \lambda = (b-b')(a-a')^{-1}\}| = \sum_{(a, b)\in E}N_{(a, b)}(\lambda).\]
    By the Cauchy-Schwarz inequality,
        \begin{equation}\label{eq:boundingQ}
        \sum_{\lambda\in \mathbb{F}_q}\left(\sum_{(a, b)\in E}N_{(a, b)}(\lambda)\right)^2\le |E|\sum_{\lambda\in \mathbb{F}_q}\sum_{(a, b)\in E}N_{a, b}(\lambda)^2.
    \end{equation}
    Note, if $(d-b)/(c-a)=\lambda$, then $(c,d)$ is contained in the line with slope $-\lambda$ that passes through the point $(a,b)$.
    From this, it is easy to see that, for any fixed $(a,b) \in E$, we have that $\sum_{\lambda \in \mathbb{F}_q} N_{a,b}(\lambda)^2$ is equal to the number of collinear triples in $E$ that contain $(a,b)$, including those triples of the form $((a,b),(c,d),(c,d))$.
    An easy consequence of the point-line incidence bound Theorem 3 in \cite{vinh1} is that the number of collinear triples of distinct points in $E$ is bounded by $O(|E|^3/q + q^2|E|)$.
    Combining this with \cref{eq:rprimebound,eq:boundingQ} leads to the upper bound
    \[|R| \ll |Y|\,|E| + |Y|^{1/2}|E|^2/q^{1/2} + |Y|^{1/2}|E|q. \]
    Combining this with the previously obtained lower bound \cref{eq:RLowerBound} leads immediately to the result.
\end{proof}
In the case that $E$ is the Cartesian product of a set with itself, a lemma due to Bourgain, Katz, and Tao \cite{BKT} leads to the following improved bound.
\begin{theorem}\label{C-Q-SE}
    Let $Y \subseteq \mathbb{F}_q\setminus \{0\}$, and $C \subseteq \F_q$, and $E=C\times C$. Then, there exists $y\in Y$ such that 
    \[\#\{ay+b\colon (a, b)\in E\}\gg \min \left\lbrace q,~\frac{|Y||E|}{q}\right\rbrace.\]
\end{theorem}
The version of the lemma that we use is \cite[Lemma 3]{GP}.
\begin{lemma}[Lemma 3, \cite{GP}]\label{BKT6}
    For $Y\subseteq \mathbb{F}_q\setminus \{0\}$ and $C \subseteq \mathbb{F}_q$, we have 
    \[\sum_{y\in Y}E^+(C, yC)\le \frac{|C|^4|Y|}{q}+q|C|^2.\]
\end{lemma}
\begin{proof}[Proof of \cref{C-Q-SE}]Proceed as in the proof of \cref{thm:generalProjectionBound}, using \cref{BKT6} for the upper bound on the size of the set $R$. \end{proof}
\subsubsection*{Improvements over prime fields}
We now move to the case of prime fields. We recall the following two improvements of Lemma \ref{BKT6}.
\begin{lemma}[Theorem 4.3, \cite{SIAM}]\label{siam}
    Let $C\subset \mathbb{F}_p$. If $|C|\le p^{1/2}$, then the number of tuples $(c_1, \ldots, c_8)$ such that 
    \[(c_1-c_2)(c_3-c_4)=(c_5-c_6)(c_7-c_8)\]
    is at most $O_\epsilon(|C|^{84/13+\epsilon})$ for any $\epsilon > 0$.
\end{lemma}
\begin{lemma}[Lemma 21, \cite{Rudnev}]\label{lm2.8}
    There is a constant $K>0$ such that the following holds.
    Let $C\subset \mathbb{F}_p$ and $Y\subset \mathbb{F}_p^*$ with $|Y|=O(|C|^2)$ and $|C|^2|Y|=O(p^2)$, then 
    \[\sum_{y\in Y}E^+(C, yC)\le K E^+(C)^{1/2}|C|^{3/2}|Y|^{3/4}.\]
\end{lemma}
As consequences, the following theorems are obtained. 

\begin{theorem}\label{C-P-SE}
    Let $Y \subset \mathbb{F}_p^*$ and $E=C\times C\subset \mathbb{F}_p^2$ with $|C|\le p^{1/2}$. Then there exists $y\in Y$ such that 
    \[\#\{ay+b\colon (a, b)\in E\}\gg \min \left\lbrace |Y|^{1/2}|C|^{10/13}, |E|\right\rbrace.\]
\end{theorem}

\begin{proof}
    Proceed as in the proof of \cref{thm:generalProjectionBound}, using \cref{siam} to bound the right side of \cref{eq:rprimebound}.
\end{proof}

\begin{theorem}
    Let $Y$ be a set in $\mathbb{F}_p\setminus \{0\}$ and $E=C\times C\subset \mathbb{F}_p^2$ with $|Y|=O(|C|^2)$ and $|C|^2|Y|=O(p^2)$. Then there exists $y\in Y$ such that 
    \[\#\{ay+b\colon (a, b)\in E\}\gg \min \left\lbrace |Y|^{1/4}\frac{|C|^{5/2}}{E^+(C)^{1/2}}, |E|\right\rbrace.\]
\end{theorem}
\begin{proof}
    Proceed as in the proof of \cref{thm:generalProjectionBound}, using \cref{lm2.8} for the upper bound on the size of the set $R$.
\end{proof}

The next theorem plays the crucial role in our proof of Theorem \ref{AAincidence}.

\begin{theorem}\label{2.4}
     Let $Y \subseteq \mathbb{F}_p^*$, and $C \subseteq \mathbb{F}_p$, and $E=C\times C$ with $|Y|=O(|C|^2)$ and $|C|^2|Y|=O(p^2)$. 
     Let $N$ and $M$ be positive real numbers such that 
    \[N^2M=\frac{|Y|^{1/4}|C|^{5/2}}{2KE^+(C)^{1/2}},\]
    where the constant $K$ is taken from \cref{lm2.8}.
    Then, there exists $y\in Y$ such that for any $E'\subset E$ with $|E'|\ge|E|/N$, we have
    \[\#\{ay+b\colon (a, b)\in E'\}> M.\]
\end{theorem}
\begin{proof}
For $y \in \F_p$ and $E_y \subseteq \F_p^2$, denote by $\langle y, E_y \rangle$ the set $\{ay+ b\colon (a, b)\in E_y\}$. 
Let $X$ be the set of $y\in Y$ such that there exists $E_y\subset E$ with $|E_y|\ge |E|/N$ and $\#\langle y, E_y \rangle\le  M$. 
For a fixed $y$, we have that $E^+(E_y,yE_y) \geq |E_y|^2/M \geq |E|^2/(N^2M)$.
Summing over all $y\in X$ and applying Lemma \ref{lm2.8}, we have 
\[\frac{|X||C|^4}{N^2M}=\frac{|X||E|^2}{N^2M} \le \sum_{y \in X}E^+(E_y,yE_y) \le KE^+(C)^{1/2}|C|^{3/2}|Y|^{3/4}.\]
This implies
\[|X|\le \frac{KN^2M|Y|^{3/4}E^+(C)^{1/2}|C|^{3/2}}{|C|^{4}}.\]
Using the fact that 
 \[N^2M=\frac{|Y|^{1/4}|C|^{5/2}}{2KE^+(C)^{1/2}}.\]
We obtain $|X|\le |Y|/2$, which completes the proof.
\end{proof}

With this result in hand, we are ready to give a proof of Theorem \ref{AAincidence}.

\begin{proof}[Proof of Theorem \ref{AAincidence}]
We aim to prove the following estimate 
\[I \vcentcolon= I(A\times B, L)\le \sqrt{8}4^{1/8}K^{1/2} |A|^{7/8}|B|^{1/2}|L|^{3/8}E^+(C)^{1/4},\]
where $K$ is the constant taken from Lemma \ref{lm2.8}.

    We proceed by induction on the size of $A$. The base case $|A|=1$ holds trivially. Indeed, $L$ consists of $|C|$ families of parallel lines. This gives $I\le |C||B|$ which is smaller than the desired bound when $E^+(C)>|C||B|^2$.

For each $a \in A$, denote by $I(a)$ the number of incidences between $L$ and $\{a\}\times B$.

If $I(a)\le I/2|A|$, then remove $a$. Denote the remaining set by $A'$. 
In total, we are left with at least $I/2$ incidences. 

If $|A'|\ge |A|/4$, then set $Y=A'$. This implies $|A|/4\le |Y|\le |A|$. Define $N$ by
\[\frac{I}{2|A|}=\frac{|L|}{N}.\]
Applying \cref{2.4}, we obtain
\[|B|\ge M=\frac{|Y|^{1/4}|C|^{5/2}}{2KN^2E^+(C)^{1/2}}.\]
This gives 
\[I\le \sqrt{8}4^{1/8}K^{1/2} |A|^{7/8}|B|^{1/2}|L|^{3/8}E^+(C)^{1/4}.\]

If $|A'|\le |A|/4$, then we apply induction to have 
\[\frac{I}{2}\le I(A'\times B, L)\le \sqrt{8}4^{1/8}K^{1/2} |A'|^{7/8}|B|^{1/2}|L|^{3/8}E^+(C)^{1/4}.\]
So 
\[I\le \sqrt{8}4^{1/8}K^{1/2} |A|^{7/8}|B|^{1/2}|L|^{3/8}E^+(C)^{1/4}.\]
This completes the proof.
\end{proof}

\section{Proof of the orthogonal projection theorem (\cref{2dimension})}

In this section, we show in detail how \cref{th:dualHalfProductIncidenceBound} implies \cref{2dimension}.
For $2a/3 \leq s \leq a$, we rely on the following combinatorial argument, which only relies on the Cauchy-Schwarz inequality.

\begin{proposition}\label{th:CS orth proj bound}
    If $E \subset \F_p^2$ with $2p^s\le |E|$, then $|T_s^{1, 2}(E)| \leq 2p^s$.
\end{proposition}
\begin{proof}
    We use the Cauchy-Schwarz inequality to bound the size of the set $\Pi_\ell = \{x,y \in E: x-y \in \ell, x \neq y\}$ for each $\ell$ such that $|\pi^\ell(E)| \leq p^s$ as follows:
    \begin{align*}
        |E|^2 &= \left(\sum_{t \in \pi_{\ell}(E)} |\{x \in E: x \in \ell + t\}|\right)^2 \\
        &\leq p^s \sum_{t \in \F_q}|\{x \in E: x \in \ell + t\}|^2 \\
        &= p^s(2|\Pi_\ell|+|E|).
    \end{align*}
    The sets $\Pi_\ell$ and $\Pi_{\ell'}$ are disjoint for $\ell \neq \ell'$, hence 
    \[\binom{|E|}{2} \geq \sum_{\ell\in T_{s}^{1,2}(E)} |\Pi_\ell| \geq 2^{-1}|T_{s}^{1, 2}|(|E|^2p^{-s} - |E|) \geq 4^{-1}|T_s^{1, 2}||E|^2p^{-s}.\]
    In the last inequality, we use the assumption that $|E| \geq 2p^s$.
    Rearranging, this implies that $|T| \leq 2p^s$, as claimed.
\end{proof}

Finally, we show how \cref{th:dualHalfProductIncidenceBound} leads to an improved bound for $s \leq 2a/3$, as sketched in the introduction.

\begin{corollary}\label{th:projectionFromIncidenceBound}
    Let $E \subset \F_p^2$ and let $p^s \ll |E|$ with $|E|p^s \ll p^2$.
    Then,
    \[|T_s^{1, 2}(E)| \ll \min(p^{5s/2}|E|^{-1}, p^{6s}|E|^{-3}). \]
\end{corollary}
\begin{proof}
    Let $L = \bigcup_{\ell \in T_s^{1, 2}(E)} \pi^\ell(E)$.
    Note that $L$ consists of $|T_s^{1, 2}(E)|$ families of parallel lines having at most $p^s$ lines each.
    Since each point of $E$ is incident to a line of $L$ for each $\ell \in T_s^{1,2}(E)$, it follows that $I(E,L) = |T_s^{1, 2}(E)|\,|E|$.
    We will apply \cref{th:dualHalfProductIncidenceBound} to obtain an upper bound on $I(E,L)$.
    By \cref{th:CS orth proj bound}, we can assume that $|T_s^{1, 2}(E)| \leq p^s$.
    Hence $|E| \, |T_s^{1, 2}(E)| \ll p^{a+s} \ll p^2$, so the hypotheses of \cref{th:dualHalfProductIncidenceBound} are satisfied.

    If $|T_s^{1, 2}(E)|\,|E| = I(L,E) \ll |T_s^{1, 2}(E)|^{2/3}p^{2s/3}|E|^{2/3} + |E|$, then $|T_s^{1, 2}(E)| \ll p^{2s} |E|^{-1}$, which is stronger than the claimed bound.
    If $|T_s^{1, 2}(E)|\,|E| = I(L,E) \ll |T_s^{1, 2}(E)|p^{s/2}|L|^{1/2} + |T_s^{1, 2}(E)|p^s$, then $|E| \ll p^s$, which contradicts the hypothesis that $p^s \ll |E|$ (for a suitable choice of the suppressed constants).
    Otherwise,
    \[|T_s^{1, 2}(E)|\,|E| = I(L,E) \ll \min(|T_s^{1, 2}(E)|^{9/10}p^{3s/5}|E|^{7/10},|T_s^{1, 2}(E)|^{7/9}p^{5s/9}|E|^{7/9}), \]
    which simplifies to the claimed bound.
\end{proof}

\cref{2dimension} follows directly from \cref{th:CS orth proj bound} together with \cref{th:projectionFromIncidenceBound}.

\section{Acknowledgments}
B. Lund was partially supported by the Institute for Basic Science (IBS-R029-C1). B. Lund and T. Pham would like to thank to the VIASM for the hospitality and for the excellent working condition.

\bibliographystyle{amsplain}

\begin{thebibliography}{10}
\bibitem{BKT}
J. Bourgain, N.H. Katz, and T. Tao, \textit{A sum-product estimate in finite fields, and
applications}, Geometric and Functional Analysis, \textbf{14}(1) (2004), 27--57.
\bibitem{Bou}
J. Bourgain, \textit{Multilinear exponential sums in prime fields under optimal entropy condition on the sources}, Geometric and Functional Analysis, \textbf{18}(5) (2009), 1477--1502.

\bibitem{BG}
P. Bright and S. Gan, \textit{Exceptional set estimates in finite fields}, Annales Fennici Mathematicim, \textbf{50} (2025), 467--481.
\bibitem{chen}
C. Chen, \textit{Projections in vector spaces over finite fields}, Annales Academiae Scientiarum Fennicae Mathematica, \textbf{43}(1) (2018), 171--185.


\bibitem{Zdvir}
 Z. Dvir, \textit{Incidence theorems and their applications}, Foundations and Trends in Theoretical Computer Science, \textbf{6}(4) (2012), 257--393.
 
\bibitem{bay}
J. K. Falconer, \textit{Hausdorff dimension and the exceptional set of projections}, Mathematika, \textbf{29}(1) (1982), 109--115.

\bibitem{tam}
J. K. Falconer, J. M. Fraser, and X. Jin, \textit{Sixty years of fractal projections}, Fractal geometry and stochastics V,  Cham: Springer International Publishing, (2015), 3--25.
\bibitem{JFraser}
J. Fraser and F. Rakhmonov, \textit{Exceptional projections in finite fields: Fourier analytic bounds and incidence geometry}, arXiv:2503.15072 (2025).
\bibitem{1gan} S. Gan,   S.  Guo, L. Guth, T.  L. J.  Harris, D.  Maldague, and H.  Wang,  \emph{On restricted projections to planes in $\mathbb {R}^ 3$}, to appear in American Journal of Mathematics, \href{https://arxiv.org/abs/2207.13844}{arXiv:2207.13844 [math.CA]}, (2024).

\bibitem{koh}
L. Q. Ham, D. T. Hoang, L. Q. Hung, D. Koh, and T. Pham, \textit{Exceptional sets for restricted families of projections in $\mathbb {F} _q^ d$}, arXiv:2510.05522 (2025).

\bibitem{HHP}
N. Hegyv\'{a}ri, F. Hennecart, and A. Plagne, \textit{A proof of two Erd\H{o}s’s conjectures on restricted addition and
further results}, Journal für die reine und angewandte Mathematik, \textbf{560} (2003), 199--220.
\bibitem{AKL}
A. Iosevich, D. Koh, and M. Lewko, \textit{Finite field restriction estimates for the paraboloid in high even dimensions}, Journal of Functional Analysis, \textbf{278}(11) (2020), 108450.
\bibitem{KF1975}
R. Kaufman and P. Mattila, \textit{Hausdorff dimension and exceptional sets of linear transformations},  Annales Academiae Scientiarum Fennicae. Series A 1, Mathematica, \textbf{1}(2) (1975), 387--392.
\bibitem{K.O.V.25} A. K\"{a}enm\"{a}ki, T. Orponen, and L. Venieri, \emph{A Marstrand-type restricted projection theorem in $\mathbb R^3$}, American Journal of Mathematics, {\bf 147}(1) (2025), 81--123.
\bibitem{KLP22}
D. Koh, S. Lee, and T. Pham, 
\textit{On the Finite Field Cone Restriction Conjecture in Four Dimensions and Applications in Incidence Geometry}, International Mathematics Research Notices, \textbf{2022} (21) (2022), 17079--17111.

 
\bibitem{Le15}
M. Lewko, \textit{New restriction estimates for the $3$-d paraboloid over finite fields}, Advances in Mathematics, \textbf{270} (2015), 457--479.

\bibitem{Lewko19}
M. Lewko, \textit{Finite field restriction estimates based on Kakeya maximal operator estimates}, Journal of the European Mathematical Society \textbf{21} (12) (2019), 3649--3707.


\bibitem{Mar1}
J. M. Marstrand, \textit{Some fundamental geometrical properties of plane sets of fractional dimensions}, Proceedings of the London Mathematical Society, \textbf{3}(1) (1954), 257--302.
\bibitem{Mat2}
P. Mattila, \textit{Hausdorff dimension, orthogonal projections and intersections with planes}, Annales Fennici Mathematici, \textbf{1}(2) (1975), 227--244.




\bibitem{SIAM}
S. Macourt, G. Petridis, I. D. Shkredov, and I. E. Shparlinski, \textit{Bounds of trilinear and trinomial exponential sums}, SIAM Journal on Discrete Mathematics, \textbf{34}(4) (2020), 2124--2136.

\bibitem{MT04} G. Mockenhaupt and T. Tao, \emph{Restriction and Kakeya phenomena for  finite fields}, Duke Mathematical Journal, \textbf{121}(1) (2004), 35--74.

\bibitem{MSS1}
A. Mohammadi and S. Stevens, \textit{Attaining the exponent $5/4$ for the sum-product problem in finite fields}, International Mathematics Research Notices, \textbf{2023}(4) (2023), 3516--3532.

\bibitem{MP}
B. Murphy and G. Petridis, \textit{Products of Differences over Arbitrary Finite Fields}, Discrete Analysis, \textbf{18} (2018), 42 pp.
 
\bibitem{MPRRS}
B. Murphy, G. Petridis, O. Roche‐Newton, M. Rudnev, and I. D. Shkredov, \textit{New results on sum‐product type growth over fields}, Mathematika, \textbf{65}(3) (2019), 588--642.
\bibitem{haitu}
Y. Peres and W. Schlag, \textit{Smoothness of projections, Bernoulli convolutions, and the dimension of exceptions}, Duke Mathematical Journal, \textbf{104}(1) (2000), 193--251.
\bibitem{GP}
G. Petridis, \textit{Products of differences in prime order finite fields}, arXiv:1602.02142 (2016).
\bibitem{Or}
T. Orponen, \textit{On the packing dimension and category of exceptional sets of orthogonal projections}, Annali di Matematica Pura ed Applicata, \textbf{194}(3) (2015), 843--880.
\bibitem{OS}
T. Orponen and P. Shmerkin, \textit{On the Hausdorff dimension of Furstenberg sets and orthogonal projections in the plane}, Duke Mathematical Journal, \textbf{172}(18) (2023), 3559--3632.
\bibitem{AR}
A. Rose, \textit{Bourgain-type projection theorems over finite fields}, arXiv:2511.08757 (2025).
\bibitem{RC}
M. Rudnev, \textit{On the number of incidences between planes and points in
three dimensions}, Combinatorica, \textbf{38}(1), 219--254, 2018.
\bibitem{Rudnev}
M. Rudnev, I.D. Shkredov, and S. Stevens, \textit{On the energy variant of the sum-product conjecture}, Revista Matem\'{a}tica Iberoamericana, \textbf{36}(1) (2019), 207--232.
\bibitem{RSH}
M. Rudnev and I. D. Shkredov, \textit{On growth rate in $ SL_2 (\mathbf {F} _p) $, the affine group and sum-product type implications}, Mathematika, \textbf{68}(3) (2022), 738--783.

\bibitem{Szemeredi}
E. Szemer\'{e}di and W.T. Trotter, \textit{Extremal problems in discrete geometry}, Combinatorica, \textbf{3}(3) (1983), 381--392.
\bibitem{frank}
S. Stevens and F. De Zeeuw, \textit{An improved point‐line incidence bound over arbitrary fields}, Bulletin of the London Mathematical Society, \textbf{49}(5) (2017), 842--858.
\bibitem{vinh1}
L. A. Vinh, \textit{The Szemer\'{e}di-Trotter type theorem and the sum-product estimate in finite fields}, European Journal of Combinatorics, \textbf{32}(8) (2011), 1177--1181.

\end{thebibliography}

\end{document}